\documentclass[preprint,12pt]{imsart}

\usepackage{amsthm,amssymb,amsmath,amscd}
\usepackage{amsfonts} 
\usepackage{pgf}
\usepackage[latin1]{inputenc}
\usepackage{verbatim}
\usepackage{mathrsfs,eucal}
\usepackage{graphicx}
\usepackage{url}
\usepackage{psfrag}
\usepackage{color}
\usepackage{pifont}
\usepackage{enumerate}
\numberwithin{equation}{section}
\usepackage{fourier}

\usepackage{natbib}
\bibpunct{(}{)}{;}{a}{,}{,}
\usepackage{hyperref}
\hypersetup{
    colorlinks=true,       
    linkcolor=medblue,          
    citecolor=medblue,        
    filecolor=medblue,      
    urlcolor=medblue           
}

\usepackage[margin=2.5cm]{geometry}
\usepackage{textcomp}
\newcommand{\ftextnumero}{{\fontfamily{txr}\selectfont \textnumero}}

\startlocaldefs
\def\journal@id{~}
\def\journal@name{~}
\def\journal@url{~}
\endlocaldefs

\definecolor{newgreen}{rgb}{0.1, 0.6, 0.1}
\definecolor{newblue}{rgb}{0.0, 0.1, 0.7}
\definecolor{ggreen}{rgb}{0.5, 0.85, 0.3}
\definecolor{rred}{rgb}{0.65, 0.2, 0.2}
\definecolor{ggray}{gray}{0.7}
\definecolor{bblue}{rgb}{0.0, 0.0, 1}
\definecolor{darkbrown}{rgb}{0.4, 0.26, 0.13}
\definecolor{medblue}{rgb}{0,0,.9}

\newtheorem{defi}{Definition}[section]
\newtheorem{lem}{Lemma}[section]
\newtheorem{cor}{Corollary}[section]
\newtheorem{pro}{Proposition}[section]
\newtheorem{theo}{Theorem}[section]
\newtheorem{exm}{Example}[section]
\newtheorem{rem}{Remark}[section]

\DeclareMathAlphabet{\mathpzc}{OT1}{pzc}{m}{it}
\newcommand{\fS}{\mathpzc{SV}}

\newcommand{\prob}{\mathbf{P}}
\newcommand{\esp}{\mathbf{E}}

\newcommand{\N}{\mathbb N}
\newcommand{\R}{\mathbb R}

\newcommand{\e}{\varepsilon}

\newcommand{\et}{\quad\mbox{and}\quad}

\newcommand{\rd}{\mathrm d}

\begin{document}
\begin{frontmatter}

\title{ON THE OCCUPANCY PROBLEM FOR A REGIME SWITCHING MODEL}
\runtitle{REGIME SWITCHING MODEL}

\begin{aug}
  \author{\fnms{Michael}  \snm{Grabchak}\thanksref{e1}\ead[label=e1]{mgrabcha@uncc.edu}},
  \author{\fnms{Mark} \snm{Kelbert}\thanksref{e2}\ead[label=e2]{mkelbert@hse.ru}},
  \and
  \author{\fnms{Quentin}  \snm{Paris}\thanksref{e3} \ead[label=e3]{qparis@hse.ru}}

\runauthor{M.~Grabchak, M.~Kelbert, and Q.~Paris}

  \address[e1]{University of North Carolina Charlotte\\
Department of Mathematics and Statistics\\
Charlotte, NC, USA. \printead{e1}}

  \address[e2]{National Research University
Higher School of Economics (HSE)\\
Faculty of Economics, Department of Statistics and Data Analysis\\
Moscow, Russia. \printead{e2}}

  \address[e3]{National Research University
Higher School of Economics (HSE)\\
Faculty of Computer Science, School of Data Analysis and Artificial Intelligence \& HDI LAB\\
Moscow, Russia. \printead{e3}}

\end{aug}

\begin{abstract} 
This article studies the expected occupancy probabilities on an alphabet. Unlike the standard situation, where observations are assumed to be independent and identically distributed (iid), we assume that they follow a regime switching Markov chain. For this model, we 1) give finite sample bounds on the occupancy probabilities, and 2) provide detailed asymptotics in the case where the underlying distribution is regularly varying. We find that, in the regularly varying case, the finite sample bounds are rate optimal and have, up to a constant, the same rate of decay as the asymptotic result.
\end{abstract}

\begin{keyword}
\kwd{occupancy problem}
\kwd{regime switching}
\kwd{Markov chain}
\kwd{regular variation}
\end{keyword}

\end{frontmatter}


\setcounter{tocdepth}{2}
{\small\tableofcontents}

\newpage

\section{Introduction} 

Let $A$ be a finite or countably infinite set and let $X=(X_{n})_{n\ge 1}$ be a discrete time $A$-valued stochastic process defined on some probability space $(\Omega, \mathcal F, \prob)$. We refer to set $A$ as the alphabet and to elements of $A$ as letters. These letters may represent different things in the context of different applications. For instance, in linguistics they may represents words in some language, while in ecology they may represent species in an ecosystem. From a general point of view, the occupancy problem (or urn scheme) is to describe the repartition of the process $(X_{n})_{n\ge 1}$ over the set $A$. In this context, two quantities of interest are
\begin{equation*}
L_n =\sum_{i=1}^{n}\mathbf 1\{X_{i}=X_{n+1}\} \et M_{n,r}=\prob\left\{\,L_n=r\,\vert\, X_{1},\dots,X_{n}\,\right\}.
\end{equation*}
These quantities are related by the fact that 
$$
\prob\{L_n=r\} = \esp M_{n,r}.
$$
In words, $L_{n}$ is the number of times that the letter observed at time $n+1$ had previously been observed and $M_{n,r}$ is the probability that, given the observations up to time $n$, the letter observed at time $n+1$ will have already been seen $r$ times. We refer to the quantities $M_{n,r}$ as the occupancy probabilities. The quantity $M_{n,0}$ is also sometimes called the missing mass. It corresponds to the probability of seeing a new letter at time $n+1$. In certain ecological contexts, it represents the probability of discovering a new species. While properties of $L_n$ and $M_{n,r}$ have been thoroughly studied in the context, where $X_1,X_2,\dots$ are independent and identically distributed (iid) random variables, we have seen no work in the literature relating to the case, where they follow a more general stochastic process.  In this paper, we give such results for a class of Markov chains, which form a regime switching model. This model expands the scope of potential applications. Moreover, it is our hope that this paper will stimulate interest in studying this problem in the context of other, more general, processes.

\subsection{Related Work} In the iid setting, the literature on the behavior of $L_{n}$, $M_{n,r}$, and related quantities is vast, see, for instance, the classic textbook \cite{Johnson:Kotz:1977}, the survey \cite{GHP07}, or recent contributions by \cite{BBO15} and \cite{DGP16}.  Applications include fields such as Ecology \citep{G53, GT56, C81, GS04}, Genomics \citep{ML02}, Language Processing \citep{CG99}, Authorship Attribution \citep{ET76, TE87,ZH07}, Information Theory \citep{OSZ04, BeBoGa16}, Computer Science \citep{Z05}, and Machine Learning \citep{BuErGa13,GZ17}. \\

We now briefly sketch several key results for the case where the random variables $X_1, X_2, \dots$ are iid with common distribution $P=\{p_a\}_{a\in A}$ on $A$. In this case, it is readily shown that 
\begin{equation*}
\prob\{L_n=r\} = \esp M_{n,r}=\binom{n}{r}\sum_{a\in A}p^{1+r}_a(1-p_a)^{n-r}.
\end{equation*}
This expression allows for a precise asymptotic analysis. Following \citet{K67}, it is understood that the main ingredients for this analysis are given by the counting measure $\boldsymbol{\nu}_{P}$ and the counting function $\nu$. These are defined, respectively, by
\begin{equation}\label{eq: count meas}
\boldsymbol{\nu}_{P}(\textrm{d}u)=\sum_{a\in A}\delta_{p_a}(\textrm{d}u)
\end{equation}
and
\begin{equation}\label{eq: count func}
\nu(\e)=\boldsymbol{\nu}_{P}([\e,1])=\sum_{k\ge 1}\mathbf 1\{p_{k}\ge\e\}, \quad 0\le \e\le 1.
\end{equation}
Next, recall that a function $\ell:(0,+\infty)\to \R$ is said to be slowly varying at $+\infty$ if for any $c>0$
\begin{eqnarray}\label{eq: sv}
\lim_{x\to +\infty}\frac{\ell(cx)}{\ell(x)}=1.
\end{eqnarray}
In this case, we write $\ell\in \fS$. With this notation, if $\nu(\e)=\boldsymbol{\nu}_{P}([\e,1])=\e^{-\alpha}\ell(1/\e)$ for some $\alpha\in(0,1)$ and some $\ell\in \fS$, then for $r\ge0$,
\begin{equation}\label{asympM}
\esp M_{n,r}\sim\frac{\alpha\Gamma(1+r-\alpha)}{r!} n^{-(1-\alpha)}\ell(n).
\end{equation}
This result is discussed, in greater detail, in the Appendix below. Non-asymptotic results are given in \cite{DGP16}. The main result of that paper is as follows.

\begin{lem}[Theorem 2.1 in \citealp{DGP16}]\label{lem:dgp}
Let $P=\{p_{k}\}_{k\ge1}$ be a probability measure on $\mathbb N_{+}$ with counting function $\nu$.  For all $n\ge 1$, all $0\le r\le n-1$, and all $0\le \e\le 1$,
$$
\prob\{L_n=r\} =\esp M_{n,r}=\binom{n}{r}\sum_{a\in A}p^{1+r}_a(1-p_a)^{n-r} \le \frac{c(r)\nu(\e)}{n}+ 2^{1+r}\binom{n}{r}\int_{0}^{\e}\nu\left(\frac u2\right)u^r\left(1-\frac u2\right)^{n-r}{\rm d}u,
$$
where 
\begin{equation}\label{cr}
c(r)=\bigg\{\begin{array}{ll}
e^{-1} &\mbox{ if }\ r=0,\\
e(1+r)/\sqrt{\pi}&\mbox{ if }\ r\ge 1.
\end{array}
\end{equation}
\end{lem}

\vspace{.4cm}

\subsection{Regime Switching Model} 

A natural extension of the iid case is to a class of regime switching Markov chains or regime switching models. In this context the elements in $A$ no longer represent letters, but entire alphabets. Each $a\in A$ represents an alphabet, which we denote by $\{a\}\times\N_+$, where $\N_+=\{1,2,\dots\}$. This alphabet has its own distribution $P_a=\{p_{a,k}\}_{k\ge1}$, and we assume that observations from each alphabet 
are iid with distribution $P_a$. However, we randomly perform transitions between alphabets following a Markov chain with transition operator $Q$. Formally, we consider a Markov chain $Z=(Z_n)_{n\ge 1}$ on the product space $\mathcal A:=A\times \N_+$ with transition operator $\mathcal Q$ defined by 
\begin{equation}\label{Q intro}
\mathcal Q((a',k'),(a,k))=Q(a',a)p_{a,k}, \ \ \ \ a,a'\in A, \ \ k,k'\in\N_+.
\end{equation}
In the interest of generality, we sometimes consider the case where transitions between alphabets do not follow a Markov chain, but a more general process. Nevertheless, our motivation comes from the case where the transitions are Markovian. Such situations can be used to describe a variety of situations, such as: 
\begin{itemize}
\item[1.] (Classics) A researcher reads documents in an antique library. The documents are written in a variety of languages (e.g.\ Latin, Greek, Hebrew, etc.). Assume that transitions between documents written in different languages follow a Markov chain. Here the regime switching Markov chain $(Z_n)_{n\ge1}$ represents the sequence of ordered pairs comprised of the word that the researcher is currently reading and the language that the current document is written in. 
In this context, the missing mass represents the probability that the next word that the researcher encounters will be one that this researcher has not  previously seen and will thus need to look up.
\item[2.] (Ecology) An ecologist is observing the animals that are found in a certain plot of forest. However, the forest has several states (e.g.\ time of day, weather, etc.) with transitions between these following a Markov chain. To understand the difference in the distribution of species found under different states, the ecologist keeps track of both the species of the observed animal and the state of the forest. 
\item[3.] (Computer Science) 
A server periodically enters a state where there is a serious hacking attempt. Assume that transitions into and out of this state follow a Markov chain.  To understand the effect of a serious hacking attempt on the number of packets that arrive, a researcher keeps track the number of packets that arrive in increments of, say, five minutes along with the state of the server in that time period.
\item[4.] (Economics) An economy can be in one of several states, e.g.\ growth, recession, inflation, etc. One can model transitions between these states using a Markov chain.  To understand the effect of the state of the economy on some economic indicator (e.g.\ the number of bank failures in a week) an economist keeps track of both the indicator and the state of the economy. 
\end{itemize}

\subsection{Organization}  The main goal of this paper is to extend the results given in Equation \eqref{asympM} and Lemma \ref{lem:dgp} from the iid case to the regime switching model. We begin by giving results for a simple class of Markov chains, which will drive this model. 
Toward this end, we introduce a useful technical result in Section \ref{section:prelim}, and then, in Section \ref{section:Finite}, we consider the case of an ergodic Markov chain on a finite state space. In Section \ref{section:sa}, we formally define the regime switching model and give extensions of Lemma \ref{lem:dgp}. In the interest of generality, most results in this section do not assume that transitions between alphabets are Markovian. However, this assumption is needed for the more detailed results. Then, in Section \ref{sec:asymp}, we extend \eqref{asympM} to the case of the regime switching model. Proofs are postposed to Section \ref{sec:proofs}. A brief review of basic properties of regularly varying distributions on an alphabet is given in the Appendix.

\subsection{Notation} Before proceeding we set up some notation. We write $\mathbf 1\{...\}$ to denote the indicator function of event $\{...\}$. For a set $A$, we write $|A|$ to denote the cardinality of $A$. For real numbers $a,b\in\mathbb R$, we write $a\vee b$ or $\max\{a,b\}$ to denote the maximum of $a$ and $b$ and we write $a\wedge b$ or $\min\{a,b\}$ to denote the minimum of $a$ and $b$. For two sequences $g(n)$ and $h(n)$ we write $g(n)\sim h(n)$ to mean $g(n)/h(n)\to1$ as $n\to\infty$. We write $\Gamma(x) = \int_0^\infty u^{x-1}e^{-u} \rd u$ for $x>0$ to denote the gamma function.

\section{Preliminaries}\label{section:prelim}

In this section we introduce a technical result, which will be useful in the sequel. Toward this end, fix a finite or countably infinite set $A$, a Markov transition operator $Q$ on $A$, and a probability measure $\mu$ on $A$. 
Let $X=(X_n)_{n\ge 1}$ be an $A$-valued random process defined on some probability space $(\Omega,\mathcal F,\prob_\mu)$ such that $X$ is a $Q$-Markov chain with initial distribution $\mu$.  We write  $\esp_{\mu}$ to denote the expectation under $\prob_{\mu}$. We write $Q^t$ to denote the $t$-step transition operator of the Markov chain. For all integers $n\ge 1$ and $a\in A$, we set 
\begin{equation}
\label{Lna}
L_n(a):=\sum_{i=1}^{n}\mathbf 1\{X_i=a\},
\end{equation}
to be the local time of Markov chain $X$ in state $a$, 
and we set
\begin{equation}
\label{Ln}
L_n:=\sum_{i=1}^{n}\mathbf 1\{X_i=X_{n+1}\}=L_n(X_{n+1})
\end{equation}
to denote the number of times that the state visited at time $n+1$ had been visited up to time $n$. \\

We now give a result, which connects the distribution of $L_n$ with that of the local times of the reversed chain. We assume that the $Q$-Markov chain $(X_n)_{n\ge1}$ is irreducible, aperiodic, positive recurrent, and has stationary distribution $\pi=(\pi_a)_{a\in A}$. We denote by $\hat X=(\hat X_n)_{n\ge 1}$ the associated reversed chain, i.e.\ an $A$-valued Markov chain with transition operator $\hat Q$ defined by
 \[
\hat Q(x,y):=\frac{\pi(y)Q(y,x)}{\pi(x)}.
\] 
It is easy to check that $\pi$ is also the stationary distribution of $\hat X$ and that the $t$-step transition operator of the reversed chain is given by
 \begin{eqnarray}\label{eq: t step rev}
\hat Q^t(x,y) = \frac{\pi(y)Q^t(y,x)}{\pi(x)}.
\end{eqnarray}
We say that the chain $X$ is reversible when $\pi(x)Q(x,y)=\pi(y)Q(y,x)$. In this case $\hat Q=Q$ and the chains $X$ and $\hat X$ have the same distribution given an initial distribution. We write $\hat L_n(a)$  to denote the local time of the reversed chain at $a$, i.e.
\[
\hat{L}_n(a):=\sum_{i=1}^{n}\mathbf 1\{\hat X_i=a\}.
\]

\begin{lem}\label{lem:taun}
Let $A$ be a finite or countably infinite set. Suppose $X=(X_n)_{n\ge 1}$ is an irreducible, aperiodic, and positive recurrent Markov chain on $A$ with stationary distribution $\pi$ and reversed chain $\hat X$. Let $\mu$ and $\eta$ be arbitrary distributions on $A$. Then, for any positive measurable function $f$ and all integers $n\ge 1$, 
\[
\esp_{\mu}\left[\frac{\eta(X_{n+1})}{\pi(X_{n+1})}f(L_n)\right]=\esp_{\eta}\left[\frac{\mu(\hat X_{n+1})}{\pi(\hat X_{n+1})}f(\hat L_{n+1}(\hat X_1)-1)\right],
\]
where, on the right-hand side, it is understood that $\eta$ is taken as the initial distribution of the reversed chain, i.e. it is the distribution of $\hat X_1$.
\end{lem}

\begin{rem}
Note, in particular, that taking $\mu=\eta=\pi$ in the above formula, and supposing the chain to be reversible, we get that for any positive measurable $f$,
\[\esp_\pi f(L_n)=\esp_\pi f(L_{n+1}(X_1)-1),\] 
so that, under $\prob_{\pi}$, $L_n$ has the same distribution as $L_{n+1}(X_1)-1$. 
\end{rem}

\section{Finite Markov Chains}\label{section:Finite}

In this section we provide a bound on $\prob_\mu\{L_n= r\}$ in the context of an ergodic Markov chain on a finite state space. This result is interesting in itself, and it will be important in the sequel because such models will drive our regime switching model. Let $X=(X_n)_{n\ge 1}$ be an irreducible and aperiodic Markov chain with finite state space $A$, transition matrix $Q$, and stationary distribution $\pi=(\pi_a)_{a\in A}$. This implies that there exists an integer $t_0\ge1$ such that
\begin{eqnarray}\label{eq:t0}
Q^{t_0}(a,b) >0 \mbox{ for all } a,b\in A.
\end{eqnarray}
From \eqref{eq: t step rev}, it follows that $\hat Q^{t_0}(a,b) >0$ for all  $a,b\in A$. Let 
\begin{eqnarray}\label{eq: lambda}
\ell= \min_{a,b}Q^{t_0}(a,b), \ \hat\ell= \min_{a,b}\hat Q^{t_0}(a,b),\ \mbox{ and } \lambda = |A|\min\{\ell,\hat\ell\},  
\end{eqnarray}
where $|A|$ is the cardinality of $A$. Note that $0<\lambda\le 1$ and that, for each $a\in A$,
\begin{eqnarray}\label{eq: bound on unif}
Q^{t_0}(a,\cdot) \ge \lambda u(\cdot) \mbox{   and   } \hat Q^{t_0}(a,\cdot) \ge \lambda u(\cdot),
\end{eqnarray}
where $u$ is the uniform distribution on $A$. By Theorem 8 in \cite{Roberts:Rosenthal:2004}, this implies that for every $a\in A$
$$
\max_{B\subset A}|Q^n(a,B)-\pi(B)| \le (1-\lambda)^{n/t_0-1}, \ \ n=1,2,\dots.
$$
This results continues to hold if $Q$ is replaced by $\hat Q$. In this context, Theorem 2 of \cite{Glynn:Ormoneit:2002} gives the following concentration inequality for $L_n(a)$. 

\begin{lem}\label{lem: concentration}
 If $\lambda>0$ and $t_0$ are such that \eqref{eq: bound on unif} holds, then for any $a\in A$, any $\gamma>0$, and any initial distribution $\mu$ we have
\[
\prob_{\mu}\left\{L_n(a)-\esp_{\pi}L_n(a)\ge n\gamma\right\} \vee \prob_{\mu}\left\{L_n(a)-\esp_{\pi}L_n(a)\le -n\gamma\right\}
\le \exp\left(-\frac{n}{2}\left(\frac{\lambda\gamma}{t_0} - \frac{2}{n} \right)^2\right),
\]
for $n>\frac{2t_0}{\lambda\gamma}$. 
\end{lem} 

Clearly, the above holds for both the chain $X$ and the reversed chain $\hat X$. Similar concentration inequalities can be obtained by applying Corollary 2.10 and Remark 2.11 in \citet{Paulin:2015}. Combining this with Lemma \ref{lem:taun} gives the following.

\begin{pro}\label{prop: main for finite}
For $n>\frac{2t_0+r\lambda  + \lambda(1-\pi_\wedge)}{\lambda\pi_\wedge}$ and any initial distribution $\mu=(\mu_a)_{a\in A}$, we have
\begin{eqnarray*}
\prob_\mu\{L_n= r\}\le \prob_\mu\{L_n\le r\} \le C \exp\left(-\frac{n}{2}\left(\frac{\lambda\pi_\wedge}{t_0} - \frac{2+(r+1)\lambda/t_0}{n} \right)^2\right),
\end{eqnarray*}
where $t_0$ and $\lambda$ are as above, $\pi_\wedge=\min_{a\in A}\pi_a$, and $C = |A|\wedge\max_{a\in A}(\mu_a/\pi_a)$.
\end{pro}

In particular, note that, when $\mu=\pi$ the constant $C=1$. It is straightforward to check that the asymptotic behavior of the upper bound is given by
$$
C  \exp\left(-\frac{n}{2}\left(\frac{\lambda\pi_\wedge}{t_0} - \frac{2+(r+1)\lambda/t_0}{n} \right)^2\right) \sim C'\exp\left(-n\frac{\lambda^2\pi_\wedge^2}{2t_0^2} \right) \mbox{  as  }  n\to\infty,
$$
where $C'= C \exp\left( t_0^{-2}\lambda\pi_\wedge (2t_0+(r+1)\lambda)\right)$. 

\begin{rem}
It may be interesting to note that Proposition \ref{prop: main for finite} gives a bound with exponential decay. This holds, in particular, for the special case, where $X_1,X_2,\dots$ are iid random variables. In comparison, Corollary 2.1 of \cite{DGP16} focuses on the iid case and only gives the bound 
$$
\prob_\mu\{L_n= r\} \le c(r)\frac{|A|}{n}, \ \ \ 0\le r\le n-1,
$$ 
where $c(r)$ is given by \eqref{cr}.
\end{rem}

The proof of Proposition \ref{prop: main for finite} depends heavily on the assumption of a finite alphabet. While concentration inequalities for the local times of Markov chains in the case of infinite alphabets are well-known and can be found in e.g.\ \cite{Glynn:Ormoneit:2002} and \citet{Paulin:2015}, there does not appear to be a simple way to transform these into bounds on $\prob_\mu\{L_n=r\}$.  The issue comes from the fact that we need $\pi_\wedge>0$, but it is always zero when $A$ is an infinite set. An interesting situation, where we are able to deal with infinite alphabets, is the regime switching model. This is  the focus of the remainder of this paper.

\section{Regime Switching Model}\label{section:sa} 

This section formally introduces the regime switching model and extends the finite sample bounds given in Lemma \ref{lem:dgp} to this case. While we are primarily interested in the case, where transitions between alphabets follow an ergodic Markov chain on a finite state space, our presentation is given in more generality.  Let $A$ be a finite or countably infinite set. For each $a\in A$, let $P_a=(p_{a,k})_{k\ge 1}$ be a probability distribution on $\N_+$. Any discrete time stochastic process $\{Y_n\}_{n\ge1}$ on $A$ can be described by a family of conditional distributions $R=\{R_n\}_{n\ge1}$, where $R_1(a) = \prob(Y_1=a)$ and for $n\ge2$
$$
R_n(a_n|a_1,a_2,\dots, a_{n-1}) = \prob(Y_n=a_n|Y_1 =a_1,Y_2=a_2,\dots,Y_{n-1}= a_{n-1}).
$$

We now introduce a process on the state space $\mathcal A:=A\times \N_+$ defined by the family of conditional distributions given by $\mathcal R=\{\mathcal R_n\}_{n\ge1}$, where $\mathcal R_1$ satisfies $\mathcal R_1((a,k))=R_1(a)p_{a,k}$ and for $n\ge2$
\begin{eqnarray}\label{eq:new R}
\mathcal R_n((a_n,k_n)|(a_1,k_1),(a_2,k_2),\dots, (a_{n-1},k_{n-1})) = R_n(a_n|a_1,a_2,\dots, a_{n-1})p_{a_n,k_n}.
\end{eqnarray}
Now, let $Z=(Z_n)_{n\ge1}$ be an $\mathcal A$-valued stochastic process governed by $\{\mathcal R_n\}_{n\ge1}$ and let $X=(X_n)_{n\ge 1}$ and $K=(K_n)_{n\ge 1}$ denote the first and second coordinate processes of $Z$, \emph{i.e.} 
\[
Z_n=(X_n,K_n),\quad n\ge 1.
\]
We will refer to the process $X$ as the underlying process.  Note, in particular, that $X$ is $A$-valued, while $K$ takes values in $\N_+$. The next result gives a more explicit description of the dynamics of the processes $X$ and $K$. 

\begin{lem}
\label{lem:jk}
In the above context, the following statements hold:
\begin{itemize}
\item[$(1)$] The process $(X_{n})_{n\ge 1}$ is governed by $\{R_n\}_{n\ge1}$.
\item[$(2)$]  For all $n\ge 1$ and for all $k\ge 1$, 
$$\prob\{K_{n}=k\,\vert\,X_{1},\dots,X_{n}\}=p_{X_{n},k},$$
where $p_{X_{n},k}$ is the random variable equal to $p_{a,k}$ on the event $\{X_{n}=a\}$.
\item[$(3)$] Conditionally on the variables $X_1,\dots,X_n$, the variables $K_1,\dots,K_n$ are independent. In particular, for all $i=1,\dots,n$ and all $k\ge 1$, 
$$\prob\{K_{i}=K_{n+1}\,\vert\,X_{1},\dots,X_{n+1},K_{n+1}\}=p_{X_{i},K_{n+1}},$$
where $p_{X_{i},K_{n+1}}$ is the random variable equal to $p_{a,k}$ on the event $\{X_{i}=a,K_{n+1}=k\}$.
\end{itemize}
\end{lem}

\vspace{.3cm}

\begin{rem}
We are motivated by the case, where $R=\{R_n\}_{n\ge1}$ represents the conditional distributions of a Markov chain with transition operator $Q$ and initial distribution $\eta$. In this case, we have: 
$R_1=\eta$ and, for $n\ge2$, 
$$
R_n(a_n|a_1,a_2,\dots, a_{n-1})   = R_2(a_n|a_{n-1}) =Q(a_{n-1},a_n).
$$
It follows that, in this case, $\mathcal R_1((a_1,k_1)=\eta(a_1)p_{a_1,k_1}$ and, for $n\ge2$, 
$$
\mathcal R_n((a_n,k_n)|(a_1,k_1),(a_2,k_2),\dots, (a_{n-1},k_{n-1})) = Q(a_{n-1},a_{n})p_{a,k},
$$
which is the Markov operator denoted by $\mathcal Q$ in \eqref{Q intro}. In this case, to emphasize the dependence on the initial distibution we will write $\prob_\eta$ for $\prob$ and $\esp_\eta$ for $\esp$. It should be noted that the subscript refers to the initial distribution of the underlying process $X$ and not of $Z$.
\end{rem}

\vspace{.5cm}

Our next results establishes a link between the quantities:
\[
\mathcal{L}_n=\sum_{i=1}^{n}\mathbf 1\{Z_i=Z_{n+1}\}\et L_n=\sum_{i=1}^{n}\mathbf 1\{X_i=X_{n+1}\}.
\]

\begin{lem}
\label{lemma:repres}
For all $n\ge 1$ and all $0\le r\le n$
\[
\prob\{\mathcal L_n=r\}= \esp \left[\binom{L_n}{r}\sum_{k=1}^{+\infty}p^{1+r}_{X_{n+1},k}(1-p_{X_{n+1},k})^{L_n-r}\right],
\]
where we take $\binom{L_n}{r}=0$ when $L_n<r$.
\end{lem}

A slight modification of Lemma \ref{lemma:repres} brings us to the main result of this section, which extends Lemma \ref{lem:dgp} from the iid case, to the regime switching case. First, we introduce some notation. For all $a\in A$, we write $\nu(a,\cdot)$ to denote the counting function of $P_{a}=(p_{a,k})_{k\in\mathbb N_+}$, which is defined, for all $0\le\e\le 1$, by 
\begin{eqnarray}\label{eq:counting meas for family}
{\nu}(a,\e)=\sum_{k\ge 1}\mathbf 1\{p_{a,k}\ge\e\}.
\label{nu2}
\end{eqnarray} 
 
\begin{theo}\label{tg}
For any $n\ge 1$ and any $0\le r\le n-1$, we have
\begin{equation}\label{eq:main}
\prob\{\mathcal L_n= r\}\le \prob\{L_n= r\}\sup_{a\in A}\sum_{k=1}^{+\infty}p^{1+r}_{a,k}+\inf_{0\le\e\le1}\{a^{n,r}(\e)+b^{n,r}(\e)\},
\end{equation}
where
\begin{eqnarray}
a^{n,r}(\e)&=&c(r)\,\esp\left[\mathbf 1\{L_n>r\}\frac{\nu(X_{n+1},\e)}{L_n}\right],
\nonumber\\
\nonumber\\
b^{n,r}(\e)&=&2^{1+r}\int_{0}^{\e}u^r\,\esp\left[\mathbf 1\{L_n>r\}\nu\left(X_{n+1},\frac u2\right)\binom{L_n}{r}\left(1-\frac u2\right)^{L_n-r}\right]{\rm d}u,
\nonumber
\end{eqnarray}
and where $c(r)$ is as in \eqref{cr}.
\end{theo}

Since, the formulation of Theorem \ref{tg} is quite general, an explicit evaluation of the coefficients $a^{n,r}(\e)$ and $b^{n,r}(\e)$ can require cumbersome computations.  More tractable formulas can be provided in a number of situations. We give several examples. 

\begin{exm}\label{tgrem1}
Consider the situation where all distributions $P_a=\{p_{a,k}\}_{k\in\N_+}$ are equal to the same distribution $P=\{p_k\}_{k\in\N_+}$ and therefore all counting functions $\nu(a,.)$ equal to the counting function $\nu$ of $P$. In this scenario, an elementary reordering of the terms in \eqref{eq:main} yields that, for any $\e\in[0,1]$, 
\begin{equation}
\label{eq:main:simple1}
\prob\{\mathcal L_n= r\}\le \sum_{m=r}^{n} C_{r,m}(\e)\,\prob\{L_n= m\},
\end{equation}
where 
\begin{equation}
\label{crr1}
C_{r,r}(\e)=C_{r,r}=\sum_{k=1}^{+\infty}p^{1+r}_{k},
\end{equation}
and, for $1+r\le m\le n$,
\begin{equation}
\label{crm1}
C_{r,m}(\e)=\frac{c(r)\nu(\e)}{m}+2^{1+r}\binom{m}{r}\int_{0}^{\e}u^r\,\nu\left(\frac u2\right)\left(1-\frac u2\right)^{m-r}{\rm d}u,
\end{equation}
where $c(r)$ is as in \eqref{cr}.
\end{exm}

\begin{exm}\label{tgrem2}
Another favorable scenario corresponds to the case where all probabilities $P_a=\{p_{a,k}\}_{k\in\N_+}$ have support contained in $\{1,\dots,M\}$ for some $M<+\infty$ independent of $a\in A$, i.e. 
\[
p_{a,k}=0 \mbox{  for  } a\in A \mbox{ and } k\ge M+1.
\] 
In this case, taking $\e=0$ on the right-hand side of \eqref{eq:main}, and noticing that $\nu(a,0)$ corresponds to the size of the support of $P_a$, yields 
\begin{equation}
\label{eq:main:simple2}
\prob\{\mathcal L_n= r\}\le \sum_{m=r}^{n} C'_{r,m}\,\prob\{L_n= m\},
\end{equation}
where 
\begin{equation}
\label{crm2}
C'_{r,r}=\sup_{a\in A}\sum_{k=1}^{M}p^{1+r}_{a,k}\et  C'_{r,m}=\frac{c(r)M}{m} \mbox{  for  }1+r\le m\le n
\end{equation}
and where $c(r)$ is as in \eqref{cr}.\\
\end{exm}

We now turn to the important situation where the distribution is regularly varying. In the iid case, the corresponding result is given in Corollary 2.2 of \citet{DGP16}. 

\begin{pro}\label{prop: RV}
Assume that for some $\alpha\in[0,1]$ and some non-increasing function $\ell\in\fS$, we have 
\[
\nu(a,\epsilon)\le \epsilon^{-\alpha}\ell(1/\epsilon)
\]
for all  $a\in A$ and all  $\epsilon\in(0,1]$. In this case,
\begin{eqnarray*}
\prob\{\mathcal L_n= r\} \le c_1(\alpha,r) \esp\left[\mathbf 1\{L_n>r\} L_n^{-(1-\alpha)}\ell(L_n)\right] + c_2(\alpha,r)\prob\{L_n=r\},
\end{eqnarray*}
where 
\begin{align*}
c_1(\alpha,r) &= c(r) + \frac{4^{1+r}}{r!}(1+r)^{1+r-\alpha}\gamma\left(1+r-\alpha,\frac{1}{2}\right),\\
c_2(\alpha,r) &= \left\{\begin{array}{ll}1&r=0,\\
\min\left\{1,p_\vee^{r+1} r^{\alpha}\ell(r) + r^{-r}\right\} &r\ge1,
\end{array}\right.
\end{align*}
$p_\vee=\sup\{p_{a,k}\}_{(a,k)\in\mathcal A}$, and $\gamma(t,x)=\int_0^x u^{t-1}e^{-u}{\rm d}u$ is the incomplete gamma function.
\end{pro}

\begin{rem}
Note that, in the case, $\alpha=1$ and $r=0$, the bound in Proposition \ref{prop: RV} is trivial since it involves $\gamma\left(0,\frac{1}{2}\right)=+\infty$. Even in the iid case, the bounds given in \cite{DGP16} are not able to deal with this case.
\end{rem}

\begin{rem}
Note that Theorem \ref{tg} and Proposition \ref{prop: RV} are quite general and hold no matter what the underlying process is. However, this generality has a cost. In particular, we still need to know quite a bit about the underlying process. In the case where the underlying process is a finite state space ergodic Markov chain, we can use Proposition \ref{prop: main for finite} and related results to get more explicit formulas. 
\end{rem}

\begin{cor}\label{cor: RV finite sample}
Assume that $|A|<\infty$ and that the underlying process is an aperiodic and irreducible Markov chain with transition operator $Q$, stationary distribution $\pi=(\pi_a)_{a\in A}$, and initial distribution $\eta$. Let $\pi_\wedge=\min_{a\in A}\pi_a$, let $t_0$ be as in \eqref{eq:t0}, and let $\lambda$ be as in \eqref{eq: lambda}. Assume further that, for some $\alpha\in[0,1]$ and some non-increasing function $\ell\in\fS$, we have 
\[
\nu(a,\epsilon)\le \epsilon^{-\alpha}\ell(1/\epsilon), \ \ a\in A, \  \epsilon\in(0,1].
\]
For any $\epsilon\in(0,\pi_\wedge)$, if $n>\frac{2t_0+r\lambda  + \lambda(1-\pi_\wedge)}{\lambda\pi_\wedge} \vee \frac{2t_0  + \lambda(1-\pi_\wedge)}{\lambda(\pi_\wedge-\epsilon)}$, then
\begin{eqnarray*}
\prob_\eta\{\mathcal L_n = r\}  \le H(n,\epsilon),
\end{eqnarray*}
where 
\begin{eqnarray*}
H(n,\epsilon) &=& c_1(\alpha,r)(n\epsilon)^{-(1-\alpha)}\ell(n\epsilon) + c_2(\alpha,r) C \exp\left(-\frac{n}{2}\left(\frac{\lambda\pi_\wedge}{t_0} - \frac{2+(r+1)\lambda/t_0}{n} \right)^2\right) \\
&&  \qquad+ c_3(\alpha,r)C \exp\left(-\frac{n}{2}\left(\frac{\lambda(\pi_\wedge-\epsilon)}{t_0} - \frac{2+\lambda/t_0}{n} \right)^2\right).
\end{eqnarray*}
Here $C$ is as in Proposition \ref{prop: main for finite},  $c_1(\alpha,r)$ and $c_2(\alpha,r)$ are as in Proposition \ref{prop: RV}, and
$$
c_3(\alpha,r) = c_1(\alpha,r)(r+1)^{-(1-\alpha)}\ell(r+1).
$$
\end{cor}

\vspace{.2cm}

It may be interesting to note that, for any $\epsilon\in(0,\pi_\wedge)$ we have
\begin{eqnarray*}
H(n,\epsilon) \sim \epsilon^{-(1-\alpha)}c_1(\alpha,r)n^{-(1-\alpha)}\ell(n) \mbox{ as }n\to\infty.
\end{eqnarray*}

\vspace{.3cm}

\section{Asymptotics For the Regime Switching Model} \label{sec:asymp}

In this section we extend  \eqref{asympM} from the iid case to the case of the regime switching model, where the underlying process is an ergodic Markov chain on a finite state space. We first define regular variation of  $P=\{p_{a,k}\}$. For a review of basic facts about regularly varying distributions on $\mathbb N_+$ we refer the reader to Appendix \ref{RV appendix}.

\begin{defi}\label{def:rvp}
We say that $P=(p_{a,k})_{(a,k)\in A\times\mathbb N_+}$ is regularly varying with index $\alpha\in[0,1]$ if there exists an $\ell\in \fS$ and a function $C:A\mapsto[0,\infty)$, which is not identically zero, such that for each $a\in A$
$$
\lim_{\e\to 0}\frac{\nu(a,\e)}{\e^{-\alpha}\ell(1/\e)}=C(a),
$$
where $\nu$ is defined as in \eqref{eq:counting meas for family}. In this case we write $P\in RV_\alpha(C,\ell)$.
\end{defi}

When $\alpha=0$, we additionally assume that there exists an $\ell_0\in \fS$ and a function $D:A\mapsto[0,\infty)$, which is not identically zero, such that for each $a\in A$
\begin{eqnarray}\label{eq: for alpha 0 in}
\lim_{\e\to0}\frac{\sum_{k\ge1}p_k \mathbf1\{p_k\le\varepsilon\}}{\varepsilon \ell_0(1/\varepsilon) }= D(a).
\end{eqnarray}

For simplicity of notation, set for $x>0$
$$
h_{\alpha,r}(x) = \left\{\begin{array}{ll}
\ell_0(x) & \alpha=0,\\
\int_x^\infty u^{-1}\ell(u)\rd u & \alpha=1, \ r=0,\\
x^{-(1-\alpha)}\ell(x) & \mbox{otherwise.}
\end{array}\right.
$$
Propositions \ref{prop:main RV limit} and \ref{prop:main RV limit to zero} imply that if $P\in RV_\alpha(C,\ell)$, then
\begin{eqnarray}\label{eq: RV}
\lim_{n\to\infty}\frac{{n\choose r}\sum_{k=1}^\infty p_{a,k}^{r+1} (1-p_{a,k})^{n-r}}{h_{\alpha,r}(n) } = F(a,r),
\end{eqnarray}
where 
\begin{eqnarray}\label{eq: F}
F(a,r) = \left\{ \begin{array}{ll} 
D(a) & \alpha=0\\
C(a) & \alpha=1, \ r=0,\\
C(a) \frac{\alpha \Gamma(r+1-\alpha)}{r!} & \mbox{otherwise.}
\end{array}\right.
\end{eqnarray}
Note that, since $|A|<\infty$, the convergence in \eqref{eq: RV} is uniform in $a$. We now give the main result for this section.

\begin{theo}\label{theprop}
In the context of the regime switching model, assume that $|A|<\infty$ and that the underlying process is an aperiodic and irreducible Markov chain with stationary distribution $\pi=(\pi_a)_{a\in A}$ and initial distribution $\eta$. Assume further that $P\in RV_\alpha(C,\ell)$ with $\alpha\in[0,1]$ (when $\alpha=0$ additionally assume that \eqref{eq: for alpha 0 in} holds) and that $\ell$ (or $\ell_0$ when $\alpha=0$) is locally bounded away from $0$ and $\infty$ on $[1,\infty)$. In this case for all $r\ge 0$ we have
$$
\lim_{n\to\infty} \frac{\prob_{\eta}\{\mathcal L_n = r\} }{h_{\alpha,r}(n) } = \sum_{a\in A} \pi_a^\alpha F(a,r),
$$
and
\begin{eqnarray*}
\lim_{n\to\infty} \frac{\prob_{\eta}\{\mathcal L_n = r\} }{\esp_\eta [\mathbf 1\{L_n>r\}h_{\alpha,r}(L_n)] } = \frac{\sum_{a\in A} \pi_a^\alpha F(a,r)}{\sum_{a\in A} \pi_a^\alpha},
\end{eqnarray*}
where $F$ is given by \eqref{eq: F}.
\end{theo}

\vspace{.3cm}

This implies that, up to a constant, we have the same asymptotics as for the upper bound in Corollary \ref{cor: RV finite sample}. It may be interesting to note that as part of the proof of the theorem, we show that for any $r\ge0$
$$
\lim_{n\to\infty} \frac{\esp_\eta [\mathbf 1\{L_n>r\}h_{\alpha,r}(L_n)] }{h_{\alpha,r}(n)} = \sum_{a\in A} \pi_a^\alpha.
$$

\section{Proofs}
\label{sec:proofs} 
\subsection{Proofs for Sections \ref{section:prelim} and \ref{section:Finite}}
\begin{proof}[Proof of Lemma \ref{lem:taun}]
Let us first prove that, for any distributions $\mu$ and $\eta$ on $A$ and any bounded function $g:A^{n+1}\to \R_+$,
\begin{equation}
\label{reverse}
\esp_{\mu}\left[\frac{\eta(X_{n+1})}{\pi(X_{n+1})}g(X_1,\dots,X_{n+1})\right]=\esp_{\eta}\left[\frac{\mu(\hat X_{n+1})}{\pi(\hat X_{n+1})} g(\hat X_{n+1},\dots,\hat X_1)\right],
\end{equation}
where, on the right-hand side, it is understood that $\eta$ is taken as the initial distribution of the reversed chain. From the definition of $\hat Q$, we obtain
\begin{align*}
&\esp_{\mu}\left[\frac{\eta(X_{n+1})}{\pi(X_{n+1})}g(X_1,\dots,X_{n+1})\right]\\
&=\sum_{x_1,\dots,x_{n+1}}\frac{\eta(x_{n+1})}{\pi(x_{n+1})}g(x_1,\dots,x_{n+1})\prob_{\mu}(X_1=x_1,\dots,X_{n+1}=x_{n+1})\\
&=\sum_{x_1,\dots,x_{n+1}}\frac{\eta(x_{n+1})}{\pi(x_{n+1})}g(x_1,\dots,x_{n+1})\mu(x_1)Q(x_1,x_2)\dots Q(x_n,x_{n+1})\\
&=\sum_{x_1,\dots,x_{n+1}}\frac{\eta(x_{n+1})}{\pi(x_{n+1})}g(x_1,\dots,x_{n+1})\mu(x_1)\frac{\pi(x_2)}{\pi(x_1)}\hat Q(x_2,x_1)\dots \frac{\pi(x_{n+1})}{\pi(x_n)}\hat Q(x_{n+1},x_n)\\
&=\sum_{x_1,\dots,x_{n+1}}\frac{\mu(x_{1})}{\pi(x_{1})}g(x_1,\dots,x_{n+1})\eta(x_{n+1})\hat Q(x_{n+1},x_n)\dots\hat Q(x_2,x_1)\\
&=\sum_{x_1,\dots,x_{n+1}}\frac{\mu(x_{1})}{\pi(x_{1})}g(x_1,\dots,x_{n+1})\prob_{\eta}(\hat X_{1}=x_{n+1},\dots,\hat X_{n+1}=x_{1})\\
&=\esp_{\eta}\left[\frac{\mu(\hat X_{n+1})}{\pi(\hat X_{n+1})} g(\hat X_{n+1},\dots,\hat X_1)\right],
\end{align*}
which proves \eqref{reverse}. Then, for any measurable positive $f$, 
\begin{align*}
\esp_{\mu}\left[\frac{\eta(X_{n+1})}{\pi(X_{n+1})}f(L_n)\right] &= \esp_{\mu}\left[\frac{\eta(X_{n+1})}{\pi(X_{n+1})}f\left(\sum_{i=1}^{n}\mathbf 1\{X_i=X_{n+1}\}\right)\right]
\nonumber\\
&= \esp_{\eta}\left[\frac{\mu(\hat X_{n+1})}{\pi(\hat X_{n+1})}f\left( \sum_{i=2}^{n+1}\mathbf 1\{\hat X_{i}=\hat X_{1}\}\right)\right]
\\
&= \esp_{\eta}\left[\frac{\mu(\hat X_{n+1})}{\pi(\hat X_{n+1})}f\left( \sum_{i=1}^{n+1}\mathbf 1\{\hat X_{i}=\hat X_1\}-1\right)\right]
\nonumber\\
&= \esp_{\eta}\left[\frac{\mu(\hat X_{n+1})}{\pi(\hat X_{n+1})}f( \hat L_{n+1}(\hat X_1)-1)\right],
\nonumber
\end{align*}
where the second line follows by applying identity \eqref{reverse} with
\[
g(x_1,\dots,x_{n+1})=f\left( \sum_{i=1}^{n}\mathbf 1\{x_{i}=x_{n+1}\}\right).
\]
This completes the proof.
\end{proof}

\begin{proof}[Proof of Proposition \ref{prop: main for finite}]
Fix $r\ge 0$ and observe that the assumption on $n$ implies that $\pi_\wedge> \frac{r}{n}$. As a result, since $\esp_{\pi}L_n(a)=n\pi_a$, we deduce from Lemma \ref{lem: concentration} that
\begin{align*}
\prob_\mu\{L_n(a)\le r\}  &= \prob_\mu\{L_n(a)-n\pi_a\le -n(\pi_a-r/n)\} \\
&\le  \prob_\mu\{L_n(a)-n\pi_a\le -n(\pi_\wedge-r/n)\} \le \exp\left(-\frac{n}{2}\left(\frac{\lambda\pi_\wedge}{t_0} - \frac{2+r\lambda/t_0}{n} \right)^2\right)
\end{align*}
when $n>2t_0/(\lambda(\pi_\wedge-r/n))$, which is equivalent to $n>(2t_0+r\lambda)/(\lambda\pi_\wedge)$. From here, we provide two bounds on $\prob_\mu\{L_n\le r\}$, which, when combined, give the desired result. First, note that
\begin{align}
\prob_\mu\{L_n\le r\} &= \sum_{a\in A} \prob_\mu\{X_{n+1} = a, L_n\le r\} \nonumber\\
&\le \sum_{a\in A} \prob_\mu\{L_n(a)\le r\} \nonumber\\
& \le |A|  \exp\left(-\frac{n}{2}\left(\frac{\lambda\pi_\wedge}{t_0} - \frac{2+r\lambda/t_0}{n} \right)^2\right).
\label{prop: main for finite:e1}
\end{align}
Next, using Lemma \ref{lem:taun} with $f(u)=\mathbf 1\{u\le r\}$, it follows that
\begin{align*}
\prob_\mu\{L_n\le r\} &= \esp_{\pi}\left[\frac{\mu(\hat X_{n+1})}{\pi(\hat X_{n+1})}\mathbf 1\{\hat L_{n+1}(\hat X_1)\le r+1\}\right]\\
&\le \max_{b\in A}\frac{\mu(b)}{\pi(b)}\prob_{\pi}\{\hat L_{n+1}(\hat X_1)\le r+1\}\\
&= \max_{b\in A}\frac{\mu(b)}{\pi(b)}\sum_{a\in A}\pi(a)\prob_{a}\{\hat L_{n+1}(a)\le r+1\},
\end{align*}
where $\prob_{a}$ is the probability measure that corresponds to the case where the initial distribution is a point-mass at $a$. Hence, using once again Lemma \ref{lem: concentration} and the fact that the stationary distribution of the reversed chain is the same as for the original chain, it follows that
\begin{equation}
\label{prop: main for finite:e2}
\prob_\mu\{L_n\le r\}\le \max_{b\in A}\frac{\mu_b}{\pi_b}  \exp\left(-\frac{n+1}{2}\left(\frac{\lambda\pi_\wedge}{t_0} - \frac{2+(r+1)\lambda/t_0}{n+1} \right)^2\right),
\end{equation}
provided $n+1>(2t_0+(r+1)\lambda)/(\lambda\pi_\wedge)$ or equivalently $n>(2t_0+r\lambda  + \lambda(1-\pi_\wedge))/(\lambda\pi_\wedge)$. The desired result follows by combining \eqref{prop: main for finite:e1} and \eqref{prop: main for finite:e2}.
\end{proof}

\subsection{Proofs for Section \ref{section:sa}}

For convenience, we sometimes denote $Y_{1\to\,m}=(Y_1,\dots, Y_{m})$ for a given process $(Y_n)_{n\ge 1}$.

\begin{proof}[Proof of Lemma \ref{lem:jk}] $(1)$ The statement follows easily from the structure of $\mathcal R$. Let $p_1$ and $p_2$ be the functions defined, for $(a,k)\in\mathcal A$, by $p_1(a,k)=a$ and $p_2(a,k)=k$. We have, 
$$
\prob(X_1 = a) = \sum_{k\ge1}\prob(Z_1=(a,k)) = \sum_{k\ge1}\mathcal R_1((a,k)) =  \sum_{k\ge1}R_1(a)p_{a,k} = R_1(a).
$$
Further, for any $n\ge 1$  and any bounded (and measurable) $f:A\mapsto \mathbb R$,
$$
\esp[f(X_{n+1})|X_{1\to n}] =\esp[ \esp[f\circ  p_1(Z_{n+1})|Z_{1\to n}]|X_{1\to n}].
$$
From here, the fact that
\begin{eqnarray*}
\esp[f\circ  p_1(Z_{n+1})|Z_{1\to n}] &=& \sum_{a\in A}\sum_{k\in\mathbb N_+}\mathcal R_{n+1}((a,k)|Z_{1},Z_2,\dots,Z_n)f\circ p_1(a,k)\\
&=&  \sum_{a\in A}\sum_{k\ge1} R_{n+1}(a|X_{1},X_2,\dots,X_n)p_{a,k} f(a) \\
&=& \sum_{a\in A} R_{n+1}(a|X_{1},X_2,\dots,X_n) f(a)
\end{eqnarray*}
implies
$$
\esp[f(X_{n+1})|X_{1\to n}] = \sum_{a\in A} R_{n+1}(a|X_{1},X_2,\dots,X_n) f(a).
$$
In particular, taking $f(a) = \mathbf 1{\{a=a'\}}$ gives
$$
\prob(X_{n+1}=a'|X_{1\to n}) =  R_{n+1}(a'|X_{1},X_2,\dots,X_n),
$$
which proves the claim.\vspace{0.2cm}\\
$(2)$ For all $n\ge 2$ all $a_1,\dots, a_n\in A$ and $k_{n}\ge 1$, 
\begin{equation}
\label{lem:jk:e1}
\prob\{K_{n}=k_{n}|X_1=a_1,\dots,X_n=a_n\} = \sum_{k_1,\dots,k_{n-1}}\frac{\prob\{Z_n=(a_n,k_n),\dots,Z_1=(a_1,k_1)\}}{\prob\{X_1=a_1,\dots,X_n=a_n\}} .
\end{equation}
Using point $(1)$ it follows that 
\[ 
\prob\{X_1=a_1,\dots,X_n=a_n\}=R_1(a_1)R_2(a_2|a_1)\cdots R(a_{n}|a_1,a_2,\dots,a_{n-1}),
\]
and that
\[
\prob\{Z_n=(a_n,k_n),\dots,Z_1=(a_1,k_1)\} =  R_1(a_1)R_2(a_2|a_1)\cdots R(a_{n}|a_1,a_2,\dots,a_{n-1}) p_{a_1,k_1}p_{a_2,k_2}\dots p_{a_n,k_n}.
\]
Combining these two identities with \eqref{lem:jk:e1}, we deduce that
\begin{align}
\prob\{K_{n}=k_{n}|X_1=a_1,\dots,X_n=a_n\}&= p_{a_n,k_n}\sum_{k_1,\dots,k_{n-1}} p_{a_1,k_1} p_{a_2,k_2}\dots p_{a_{n-1},k_{n-1}}
\nonumber\\
&=p_{a_n,k_n},
\nonumber
\end{align}
where the last identity follows from the fact that $\sum_k p_{a,k}=1$. The case where $n=1$ is similar.
 \vspace{0.2cm}\\
$(3)$ For any $a_1,\dots, a_n\in A$ and any $k_1,\dots k_n\in \mathbb N_+$,
\begin{align*}
\prob\{K_1=k_1,\dots,K_{n}=k_{n}|X_1=a_1,\dots,X_n=a_n\} &= \prod_{i=1}^{n}p_{a_i,k_i}\\
&=\prod_{i=1}^{n}\prob\{K_i=k_i|X_1=a_1,\dots,X_i=a_i\},
\end{align*}
were the first identity follows by arguments similar to those used in the proof of point $(2)$ and the second follows directly from point $(2)$. Finally, the proof that, for $i=1,2,\dots,n$
\[
\prob \left\{K_{i}=K_{n+1}\,\vert\,X_{1},\dots,X_{n+1},K_{n+1}\right\}=p_{X_{i},K_{n+1}}
\] 
is very similar and is omitted for brevity.
\end{proof}

\begin{proof}[Proof of Lemma \ref{lemma:repres}] 
Fix $n\ge 1$ and $0\le r\le n$. Since $\{\mathcal L_n=r\}\subset\{L_n\ge r\}$, we have
\begin{equation}
\label{th:tge1}
\prob\{\mathcal L_n=r\}= \prob \{L_n\ge r,\mathcal L_n=r\}.
\end{equation}
Noticing that the variable $L_n$ is $\sigma(X_{1\to\,n+1})$-measurable by construction, we obtain
\begin{align}
\prob\{\mathcal L_n=r\} &=\prob\left\{L_n\ge r,\,\sum_{i=1}^{n}\mathbf 1\{X_i=X_{n+1},K_i=K_{n+1}\}=r\right\},
\nonumber\\
&=\esp\left[\mathbf 1\{L_n\ge r\}\prob\left(\,\sum_{i=1}^{n}\mathbf 1\{X_i=X_{n+1},K_i=K_{n+1}\}=r\,\vert\,K_{n+1},X_{1\to\,n+1}\right)\right].
\nonumber
\end{align}
Conditionally on $K_{n+1}$ and $X_{1\to\,n+1}$ the variables $K_{1},\dots,K_{n}$ are, according to point $(3)$ of Lemma \ref{lem:jk}, independent and satisfy 
\begin{equation}
\label{th:tge2}
\prob\{K_{i}=K_{n+1}\,\vert\,X_{1\to\,n+1},K_{n+1}\}=p_{X_{i},K_{n+1}}.
\end{equation}
As a result, conditionally on $K_{n+1}$ and $X_{1\to\,n+1}$, the variable
$$\sum_{i=1}^{n}\mathbf 1\{X_i=X_{n+1},K_i=K_{n+1}\},$$
follows a Binomial distribution with parameters $L_n$ and $p_{X_{n+1},K_{n+1}}$. Hence, we obtain
\begin{align}
\prob\{\mathcal L_n=r\}&=\esp\left[\mathbf 1\{L_n\ge r\}\binom{L_n}{r}p^{r}_{X_{n+1},K_{n+1}}(1-p_{X_{n+1},K_{n+1}})^{L_n-r}\right]
\nonumber\\
&=\esp\left[\mathbf 1\{L_n\ge r\}\binom{L_n}{r}\esp\left[p^{r}_{X_{n+1},K_{n+1}}(1-p_{X_{n+1},K_{n+1}})^{L_n-r}\,\vert\,X_{1\to\,n+1}\right]\right]
\nonumber\\
&=\esp\left[\mathbf 1\{L_n\ge r\}\binom{L_n}{r}\sum_{k\ge1}p^{1+r}_{X_{n+1},k}(1-p_{X_{n+1},k})^{L_n-r}\right]
\nonumber,
\end{align}
where the last line follows from point $(2)$ of Lemma \ref{lem:jk}. 
\end{proof}

\begin{proof}[Proof of Theorem \ref{tg}]
From Lemma \ref{lemma:repres} it follows that
\begin{eqnarray*}
\prob\{\mathcal L_n=r\} &=& \esp\left[\mathbf 1\{L_n= r\}\binom{L_n}{r}\sum_{k\ge1}p^{1+r}_{X_{n+1},k}(1-p_{X_{n+1},k})^{L_n-r}\right] \\
&&+\ \esp\left[\mathbf 1\{L_n> r\}\binom{L_n}{r}\sum_{k\ge1}p^{1+r}_{X_{n+1},k}(1-p_{X_{n+1},k})^{L_n-r}\right]\\
&=:& A_1(n) + A_2(n).
\end{eqnarray*}
Note that 
$$
A_1(n) =  \esp\left[\mathbf 1\{L_n=r\}\sum_{k\ge1}p^{1+r}_{X_{n+1},k}\right] \le \prob\{L_n= r\}\sup_{a\in A}\sum_{k=1}^{+\infty}p^{1+r}_{a,k}.
$$
Now, using Lemma \ref{lem:dgp} inside the expectation yields
\begin{align}
A_2(n) & \le \esp\left[\mathbf 1\{L_n>r\} \inf_{0\le\e\le 1}\{\alpha^{n,r}(\e)+\beta^{n,r}(\e)\}\right]
\label{th:tge4},
\end{align}
where we have denoted 
\begin{align}
\alpha^{n,r}(\e) & =\frac{c(r)\nu(X_{n+1},\e)}{L_n},
\nonumber\\
\beta^{n,r}(\e) & = 2^{1+r}\binom{L_n}{r}\int_{0}^{\e}\nu\left(X_{n+1},\frac u2\right)u^r\left(1-\frac u2\right)^{L_n-r}{\rm d}u.
\nonumber
\end{align}
Finally, observing the fact that
\begin{align*}
A_2(n) & \le \esp\left[ \inf_{0\le\e\le 1}\{\mathbf 1\{L_n>r\}\alpha^{n,r}(\e)+\mathbf 1\{L_n>r\}\beta^{n,r}(\e)\}\right]\\
& \le \inf_{0\le\e\le 1}\left\{\esp\left[\mathbf 1\{L_n>r\}\alpha^{n,r}(\e)\right]+\esp\left[\mathbf 1\{L_n>r\}\beta^{n,r}(\e)\right]\right\}\\
& \le \inf_{0\le\e\le 1}\left\{ a^{n,r}(\e)+b^{n,r}(\e)\right\},
\end{align*}
gives the result.
\end{proof}

\begin{proof}[Proof of Proposition \ref{prop: RV}]
By Lemma \ref{lemma:repres}, we have
\begin{align*}
\prob\{\mathcal L_n=r\}  &= \esp\left[\mathbf 1\{L_n>r\} \binom{L_n}{r}\sum_{k\ge1}p^{1+r}_{X_{n+1},k}(1-p_{X_{n+1},k})^{L_n-r}\right] \\
&+  \esp\left[\mathbf 1\{L_n=r\} \sum_{k\ge1}p^{1+r}_{X_{n+1},k}\right]\\
& =: E_1 + E_2.
\end{align*}
Corollary 2.2 from \citet{DGP16} implies that
$$
E_1 \le c_1(\alpha,r) \esp\left[\mathbf1\{L_n>r\} L_n^{-(\alpha-1)}\ell(L_n)\right] .
$$
From here, the results follows in the case where $r=0$ from the fact that $ \sum_{k\ge1}p_{X_{n+1},k}=1$. Now, assume that $r\ge1$. Taking $\epsilon=1/r$ in (2.4) of \citet{DGP16} implies
$$
E_2 \le \left(p_\vee^{r+1} r^{\alpha}\ell(r) + r^{-r}\right) \prob_{\eta}\{L_n=r\}.
$$
On the other hand, since $\sum_{k\ge1}p^{1+r}_{X_{n+1},k}\le \sum_{k\ge1}p_{X_{n+1},k}=1$, we also have
$$
E_2 \le \prob_{\eta}\{L_n=r\}.
$$
This completes the proof.
\end{proof}

\begin{proof}[Proof of Corollary \ref{cor: RV finite sample}]
Fix $\epsilon\in(0,\pi_\wedge)$, let
$$
A(n) = \{r+1\le L_n < n\epsilon\} \mbox{ and }
B(n) = \{L_n\ge (n\epsilon) \vee (r+1)\},
$$
and note that $A(n)\cup B(n)=\{r+1\le L_n\}$. We can write
\begin{align*}
 \esp_{\eta}\left[\mathbf 1\{L_n>r\} L_n^{-(1-\alpha)}\ell(L_n)\right] &=  \esp_{\eta}\left[\mathbf 1_{A(n)} L_n^{-(1-\alpha)}\ell(L_n)\right] +  \esp_{\eta}\left[\mathbf 1_{B(n)} L_n^{-(1-\alpha)}\ell(L_n)\right] = E_1+E_2.
\end{align*}
Now note that
\begin{eqnarray*}
E_1 \le (r+1)^{-(1-\alpha)}\ell(r+1)\prob_\eta\{r+1\le L_n< n\epsilon\}
\end{eqnarray*}
and
\begin{eqnarray*}
E_2 \le (n\epsilon)^{-(1-\alpha)}\ell(n\epsilon) \prob_\eta\{L_n\ge n\epsilon\}.
\end{eqnarray*}
Combining this with Proposition \ref{prop: RV} gives
\begin{eqnarray*}
\prob_{\eta}\{\mathcal L_n = r\} &\le&   \inf_{\epsilon\in(0,\pi_\wedge)}\left\{ c_1(\alpha,r)(n\epsilon)^{-(1-\alpha)}\ell(n\epsilon) \prob_\eta\{L_n\ge  n\epsilon\}  + c_2(\alpha,r)\prob_{\eta}\{L_n=r\}\right.\\
&&\left. \qquad
 + c_3(\alpha,r)\prob_\eta\{r+1\le L_n< n\epsilon\}  \right\}.
\end{eqnarray*}
From here, the result follows by applying Proposition \ref{prop: main for finite}.
\end{proof}

\subsection{Proofs for Section \ref{sec:asymp}}

To prove Theorem \ref{theprop}, we begin with two technical results. 

\begin{lem}\label{lemma: technical}
Let $(X_n)_{n\ge1}$ be an irreducible and aperiodic Markov chain on a finite state space $A$ and with stationary distribution $\pi = (\pi_a)_{a\in A}$. Let $\pi_\wedge=\min_{a\in A}\pi_a$ and let $L_n= \sum_{k=1}^n \mathbf 1\{X_k=X_{n+1}\}$.\\
1. For any $\beta\in\mathbb R$, any $\epsilon\in[0,\pi_\wedge)$, any $r>0$, and any initial distribution $\eta$ we have
\begin{eqnarray*}
\lim_{n\to\infty} n^{\beta} \prob_{\eta}\left\{\frac{L_n}{n}\le\epsilon \right\} =0
\end{eqnarray*}
and
$$
\lim_{n\to\infty}n^{\beta} \prob_{\eta}\{L_n=r\}=0.
$$
2.  If $\alpha\in[0,1]$ and $\ell\in\fS$, then, with probability $1$,
\begin{eqnarray}\label{eq: tech3}
\lim_{n\to\infty}\left( \frac{L_n^{-(1-\alpha)}\ell(L_n)}{n^{-(1-\alpha)}\ell(n)} - \pi_{X_{n+1}}^{-(1-\alpha)} \right)= 0
\end{eqnarray}
and for any $r\ge0$ and any initial distribution $\eta$
$$
\lim_{n\to\infty} \frac{\esp_{\eta}[\mathbf 1\{L_n>r\}L_n^{-(1-\alpha)}\ell(L_n)]}{n^{-(1-\alpha)}\ell(n)}  = \sum_{a\in A} \pi_a^\alpha .
$$
\end{lem}

\begin{proof}
The first part follows immediately from the exponential bound in Proposition \ref{prop: main for finite}. We now turn to the second part. For ease of notation, set $h(x) = x^{-(1-\alpha)}\ell(x)$. 
Since the Markov chain is irreducible and aperiodic on a finite state space, it is recurrent and hence $\lim_{n\to\infty}L_n=\infty$ with probability $1$. Further, it satisfies the strong law of large numbers, which mean that for each $a\in A$, if $L_n(a) = \sum_{k=1}^n \mathbf 1\{X_k=a\}$, then $\lim_{n\to\infty} L_n(a)/n=\pi_a$ with probability $1$. Since $A$ is a finite set, with probability $1$, this convergence can be taken to be uniform in $a$. Let  $\Omega_0\subset\Omega$ with $\prob_\eta(\Omega_0)=1$, such that for any $\omega\in\Omega_0$ we have $\lim_{n\to\infty}L_n(\omega)=\infty$ and for any $\epsilon>0$ there exists an $N'_\epsilon(\omega)$ such that if $n\ge N'_\epsilon(\omega)$ then
$$
\left|\frac{L_n(\omega)}{n} - \pi_{X_{n+1}(\omega)}\right|<\epsilon.
$$
Now fix $\epsilon>0$ and $\omega\in\Omega_0$.  There exists an $N_\epsilon(\omega)>0$ such that if $n\ge N_\epsilon(\omega)$ then $\left|\frac{L_n(\omega)}{n} - \pi_{X_{n+1}(\omega)}\right|<.5\pi_\wedge$ and
$$
\left|\pi_{X_{n+1}(\omega)}^{-(1-\alpha)} -  \left(\frac{L_n(\omega)}{n}\right)^{-(1-\alpha)}\right| < \epsilon/2.
$$
Further, by the uniform convergence theorem for regularly varying functions, see e.g.\ Proposition 2.4 in \cite{Resnick:2007}, there is a $T_\epsilon$ such that, for any $x\in(.5\pi_{\wedge},1]$ and any $t\ge T_\epsilon$
$$
\left|\frac{h(xt)}{h(t)} - x^{-(1-\alpha)}\right| < \epsilon/2.
$$
Since $\frac{L_n}{n}\le 1$, it follows that, for $n\ge \max\{N_\epsilon(\omega),T_\epsilon\}$,
$$
\left|\frac{h\left(\frac{L_n(\omega)}{n} n\right)}{h(n)} - \pi_{X_{n+1}(\omega)}^{-(1-\alpha)}\right|\le   \left|\frac{h\left(\frac{L_n(\omega)}{n} n\right)}{h(n)} - \left(\frac{L_n(\omega)}{n}\right)^{-(1-\alpha)}\right| + \left|\pi_{X_{n+1}(\omega)}^{-(1-\alpha)} -  \left(\frac{L_n(\omega)}{n}\right)^{-(1-\alpha)}\right| < \epsilon,
$$
which proves \eqref{eq: tech3}. We now turn to the last part. Fix $\epsilon\in(0,\pi_\wedge)$ and let
$$
A(n) = \{r+1\le L_n < n\epsilon\} \mbox{ and }
B(n) = \{L_n\ge (n\epsilon) \vee (r+1)\}.
$$
Note that $A(n)\cup B(n)=\{r+1\le L_n\}$. We can write
\begin{eqnarray*}
&&\frac{\esp_{\eta}[\mathbf 1\{L_n>r\}L_n^{-(1-\alpha)}\ell(L_n)]}{n^{-(1-\alpha)}\ell(n)} \\
&&\qquad\qquad =  \frac{\esp_{\eta}[ \mathbf 1_{A(n)}L_n^{-(1-\alpha)}\ell(L_n)]}{n^{-(1-\alpha)}\ell(n)}+ \frac{\esp_{\eta}[\mathbf 1_{B(n)}L_n^{-(1-\alpha)}\ell(L_n)]}{n^{-(1-\alpha)}\ell(n)} \\
&&\qquad\qquad=: E_{A}(n) + E_{B}(n).
\end{eqnarray*}
Fix $\delta>0$, by the Potter bounds (see e.g.\ Theorem 1.5.6 in \cite{Bingham:Goldie:Teugels:1987}), there exists a $K_\delta>0$ such that
\begin{eqnarray*}
E_A(n) \le  K_\delta  \esp_{\eta}\left[\mathbf 1_{A(n)} \left( \frac{L_n}{n}\right)^{-(1-\alpha)-\delta} \right] 
\le  K_\delta  \prob_{\eta}(A(n)) n^{1-\alpha+\delta} \to0,
\end{eqnarray*}
where the convergence follows by the first part of this lemma. Similarly, 
\begin{eqnarray*}
\frac{\mathbf 1_{B(n)}L_n^{-(1-\alpha)}\ell(L_n)}{n^{-(1-\alpha)}\ell(n)} &\le& K_\delta \mathbf 1_{B(n)} \left( \frac{L_n}{n}\right)^{-(1-\alpha)-\delta} \le K_\delta \epsilon^{-(1-\alpha)-\delta}.
\end{eqnarray*}
Combining this with the fact that $\pi_{X_{n+1}}^{-(1-\alpha)}$ is bounded means that we can use dominated convergence to get
\begin{eqnarray*}
 \lim_{n\to\infty}E_B(n) &=& \lim_{n\to\infty}(E_{B} (n) + \esp_{\eta}[\pi_{X_{n+1}}^{-(1-\alpha)}] - \esp_{\eta}[\pi_{X_{n+1}}^{-(1-\alpha)}])\\
&=& \esp_{\eta}\left[\lim_{n\to\infty} \left(\frac{\mathbf 1_{B(n)}L_n^{-(1-\alpha)}\ell(L_n)}{n^{-(1-\alpha)}\ell(n)} - \pi_{X_{n+1}}^{-(1-\alpha)}\right)\right] +  \lim_{n\to\infty} \esp_{\eta}[\pi_{X_{n+1}}^{-(1-\alpha)}] \\
&=&\lim_{n\to\infty} \esp_{\eta}[\pi_{X_{n+1}}^{-(1-\alpha)}] = \esp_{\pi}[\pi_{X_{1}}^{-(1-\alpha)}]  =\sum_{a\in A} \pi_a^\alpha,
\end{eqnarray*}
where the third equality follows from \eqref{eq: tech3} and the fact that, with probability $1$, there exists a (random) $N$ such that $\mathbf 1_{B(n)}=1$ for all $n\ge N$, and the fourth equality follows by  the fact that the distribution of $X_{n}$ converges weakly to $\pi$, Skorokhod's representation theorem, and dominated convergence. 
\end{proof}

\begin{lem}\label{lemma: general lim}
Let $|A|<\infty$ and let $P\in RV_\alpha(C,\ell)$. When $\alpha=0$ assume, in  addition, that \eqref{eq: for alpha 0 in} holds.

1. Let $X_n$ be any sequence of $A$-valued random variables and let $N_n$ is a sequence of $\mathbb N$-valued random variables such that, with probability $1$, $N_n\to\infty$ as $n\to\infty$. With probability $1$,
$$
\lim_{n\to\infty}\left( \frac{{N_n\choose r}\sum_{k=1}^\infty p_{X_{n+1},k}^{r+1} (1-p_{X_{n+1},k})^{N_n-r}}{h_{\alpha,r}(N_n)} - F(X_{n+1}, r)\right)=0.
$$
2.  Let $X=(X_k)_{k\ge1}$ be an irreducible and aperiodic Markov chain with state space $A$ and stationary distribution $\pi=(\pi_a)_{a\in A}$. If $L_n= \sum_{k=1}^n \mathbf 1\{X_k=X_{n+1}\}$, then, with probability $1$,
$$
\lim_{n\to\infty}\left(\frac{\binom{L_n}{r}\sum_{k=1}^\infty p^{1+r}_{X_{n+1},k}(1-p_{X_{n+1},k})^{L_n-r}}{h_{\alpha,r}(n)}- \pi_{X_{n+1}}^{-(1-\alpha)}F(X_{n+1},r)\right) =0.
$$
\end{lem}

Note that, in the first part, the sequences $\{X_{n}\}$ and $\{N_n\}$ may be dependent or independent.

\begin{proof}
We begin with the first part. Let $\Omega_0\in\mathcal F$ be a set with $\prob(\Omega_0)=1$ such that, for any $\omega\in\Omega_0$, $N_n(\omega)\to\infty$. Fix $\epsilon>0$ and $\omega\in\Omega_0$. Since \eqref{eq: RV} holds uniformly in $a$, it follows that there is an $M_\epsilon>0$ such that for all $m\ge M_\epsilon$ and all $n\ge1$ we have
\begin{eqnarray*}
\left|\frac{{m\choose r}\sum_{k=1}^\infty p_{X_{n+1}(\omega),k}^{r+1} (1-p_{X_{n+1}(\omega),k})^{m-r}}{ h_{\alpha,r}(m)}- F(X_{n+1}(\omega),r) \right|<\epsilon.
\end{eqnarray*}
Now let $M'_\epsilon(\omega)>0$ be a number such that, if $n\ge M'_\epsilon(\omega)$ then $N_n(\omega)\ge  M_\epsilon$. For all such $n$, the above holds with $N_n(\omega)$ in place of $m$. From here, the first part follows. For the second part, we have
\begin{eqnarray*}
&&\frac{\binom{L_n}{r}\sum_{k=1}^\infty p^{1+r}_{X_{n+1},k}(1-p_{X_{n+1},k})^{L_n-r}}{h_{\alpha,r}(n)}- \pi_{X_{n+1}}^{-(1-\alpha)}F(X_{n+1},r) \\
&& \quad = \left(\frac{h_{\alpha,r}(L_n)}{h_{\alpha,r}(n)} - \pi_{X_{n+1}}^{-(1-\alpha)}\right)\left( \frac{\binom{L_n}{r}\sum_{k=1}^\infty p^{1+r}_{X_{n+1},k}(1-p_{X_{n+1},k})^{L_n-r}}{h_{\alpha,r}(L_n)}- F(X_{n+1},r)\right) \\
&&\qquad+\pi_{X_{n+1}}^{-(1-\alpha)} \left( \frac{\binom{L_n}{r}\sum_{k=1}^\infty p^{1+r}_{X_{n+1},k}(1-p_{X_{n+1},k})^{L_n-r}}{h_{\alpha,r}(L_n)}- F(X_{n+1},r)\right) \\
&&\qquad + F(X_{n+1},r) \left( \frac{h_{\alpha,r}(L_n)}{h_{\alpha,r}(n)} - \pi_{X_{n+1}}^{-(1-\alpha)} \right).
\end{eqnarray*}
Since the Markov chain $X$ is irreducible on a finite state space, all of its states are recurrent and hence $\lim_{n\to\infty}L_n=\infty$ with probability $1$. Thus, by the first part of this lemma, the fact that $\max_{a\in A}F(a,r)<\infty$, and the fact that $\pi_{X_{n+1}}^{-(1-\alpha)}\le \left(\min_a\pi_a\right)^{-(1-\alpha)}$, it suffices to show that, with probability $1$,
$$
\lim_{n\to\infty}\left(  \frac{h_{\alpha,r}(L_n)}{h_{\alpha,r}(n)} - \pi_{X_{n+1}}^{-(1-\alpha)} \right)= 0,
$$
which holds by Lemma \ref{lemma: technical}.
\end{proof}

\begin{proof}[Proof of Theorem \ref{theprop}]
Note that, by Lemma \ref{lemma:repres}
\begin{eqnarray*}
\prob_\eta\{\mathcal L_n=r\} &=& \esp_{\eta}\left[\mathbf 1\{L_n>r\}\binom{L_n}{r}\sum_{k\ge1}p^{1+r}_{X_{n+1},k}(1-p_{X_{n+1},k})^{L_n-r}\right]\\
&&\qquad + \esp_{\eta}\left[\mathbf 1\{L_n=r\}\sum_{k\ge1}p^{1+r}_{X_{n+1},k}\right] =:E_1+E_2.
\end{eqnarray*}

We begin with $E_2$. Since $\sum_{k\ge1}p^{1+r}_{X_{n+1},k}\le \sum_{k\ge1}p_{X_{n+1},k}=1$, it follows that
\begin{eqnarray*}
\frac{E_2}{h_{\alpha,r}(n)}=\frac{\esp_{\eta}\left[\mathbf 1\{L_n=r\}\sum_{k\ge1}p^{1+r}_{X_{n+1} ,k}\right]}{h_{\alpha,r}(n)} &\le& \frac{P_{\eta}\{L_n=r\}}{h_{\alpha,r}(n)} \to0,
\end{eqnarray*}
where the convergence follows by Lemma \ref{lemma: technical}. We next turn to $E_1$. Note that \eqref{eq: RV}, the fact that $|A|<\infty$, and the fact that $h_{\alpha,r}$ is locally bounded implies that there is a constant $K'>0$ depending only on $r$ with
\begin{eqnarray*}
\frac{\binom{L_n}{r}\sum_{k\ge1}p^{1+r}_{X_{n+1},k}(1-p_{X_{n+1},k})^{L_n-r}}{h_{\alpha,r}(n)} 
&\le& K'\frac{h_{\alpha,r}(L_n)}{h_{\alpha,r}(n)}\\
&\le& H_\delta K'\left( \frac{L_n}{n}\right)^{-(1-\alpha)-\delta},
\end{eqnarray*}
for any $\delta>0$ and some $H_\delta>1$. Here the second inequality follows by the Potter bounds, see e.g.\ Theorem 1.5.6 in \cite{Bingham:Goldie:Teugels:1987}. For simplicity, set $K_\delta = K' H_\delta$. Fix $\epsilon\in(0,\pi_\wedge)$ and let
$$
A(n) = \{r+1\le L_n < n\epsilon\} \mbox{ and }
B(n) = \{L_n \ge (n\epsilon) \vee (r+1)\}.
$$
Note that $A(n)\cup B(n)=\{r+1\le L_n\}$. We can write
\begin{eqnarray*}
E_1&=& \esp_{\eta}\left[\mathbf 1_{A(n)}\binom{L_n}{r}\sum_{k\ge1}p^{1+r}_{X_{n+1},k}(1-p_{X_{n+1},k})^{L_n-r}\right]\\
&&\qquad + \esp_{\eta}\left[\mathbf 1_{B(n)}\binom{L_n}{r}\sum_{k\ge1}p^{1+r}_{X_{n+1},k}(1-p_{X_{n+1},k})^{L_n-r}\right] =: E_{1A} + E_{1B}.
\end{eqnarray*}
By Lemma \ref{lemma: technical}, we have
\begin{eqnarray*}
\frac{E_{1A} }{h_{\alpha,r}(n)} 
&\le&  K_\delta  \esp_{\eta}\left[\mathbf 1_{A(n)} \left( \frac{L_n}{n}\right)^{-(1-\alpha)-\delta} \right] \\
&\le&  K_\delta  \prob_{\eta}(A(n)) n^{1-\alpha+\delta} \to0.
\end{eqnarray*}
Similarly, 
\begin{eqnarray*}
\frac{\mathbf 1_{B(n)}\binom{L_n}{r}\sum_{k\ge1}p^{1+r}_{X_{n+1},k}(1-p_{X_{n+1},k})^{L_n-r}}{h_{\alpha,r}(n)} &\le& K_\delta \mathbf 1_{B(n)} \left( \frac{L_n}{n}\right)^{-(1-\alpha)-\delta}\\
&\le& K_\delta \epsilon^{-(1-\alpha)-\delta}.
\end{eqnarray*}
Combining this with the fact that $\pi_{X_{n+1}}^{-(1-\alpha)}F(X_{n+1},r)$ is bounded for fixed $r$ means that we can use dominated convergence to get
\begin{eqnarray*}
 \lim_{n\to\infty}\frac{E_{1B} }{h_{\alpha,r}(n)} &=& 
\esp_{\eta}\left[\lim_{n\to\infty} \left(\mathbf 1_{B(n)} \frac{\binom{L_n}{r}\sum_{k\ge1}p^{1+r}_{X_{n+1},k}(1-p_{X_{n+1},k})^{L_n-r}}{h_{\alpha,r}(n)} - \pi_{X_{n+1}}^{-(1-\alpha)} F(X_{n+1},r) \right)\right] \\
&&\qquad+\lim_{n\to\infty} \esp_{\eta}[\pi_{X_{n+1}}^{-(1-\alpha)}F(X_{n+1},r)] \\
&=& \lim_{n\to\infty} \esp_{\eta}[\pi_{X_{n+1}}^{-(1-\alpha)}F(X_{n+1},r)] = \esp_{\pi}[\pi_{X_{1}}^{-(1-\alpha)}F(X_{1},r)]  = \sum_{a\in A} \pi_a^\alpha F(a,r),
\end{eqnarray*}
where the second equality follows from Lemma \ref{lemma: general lim} and the fact that, with probability $1$, there exists a (random) $N$ such that $\mathbf 1_{B(n)}=1$ for all $n\ge N$. The third equality, follows from the fact that the distribution of $X_{n}$ converges weakly to $\pi$, Skorokhod's representation theorem, and dominated convergence. This gives the first part of Theorem \ref{theprop}. The second part follows from the first and Lemma \ref{lemma: technical}.
\end{proof}

\appendix

\section{Regular variation}\label{RV appendix} 

In this appendix, we briefly review several basic facts about regularly varying distributions on $\mathbb N_+$. First, we recall that for a probability measure $P=\{p_{k}\}_{k\in\mathbb N_+}$ on $\mathbb N_+$, the counting measure $\boldsymbol{\nu}_P$ is defined by \eqref{eq: count meas} and the counting function $\nu$ is defined by \eqref{eq: count func}.
  
\begin{defi}
A probability distribution $P=\{p_{k}\}_{k\ge 1}$ with counting function $\nu$ is said to be regularly varying, with exponent $\alpha\in[0,1]$,  if 
\[\lim_{\e\to 0}\frac{\nu(\e)}{\e^{-\alpha}\ell(1/\e)}=1,\]
for function $\ell\in\fS$. 
In this case, we write $P\in RV_{\alpha}(\ell)$.
\end{defi}

To motivate this definition, we recall the following fact from \cite{GHP07}. For $\alpha\in(0,1)$, we have $P\in RV_{\alpha}(\ell)$ if and only if
$$
p_k\sim k^{-1/\alpha} \ell^*(k) \mbox{ as } k\to\infty
$$
for some $\ell^*\in\fS$, which is, in general, different from $\ell$.  When $\alpha=0$, a sufficient condition for $P\in RV_{\alpha}(\ell)$ is that there exists an $\ell_0\in\fS$ with
\begin{eqnarray}\label{eq: for alpha 0}
\int_{[0,\varepsilon]} x \boldsymbol{\nu}_P(\rd x) =\sum_{k\ge1}p_k \mathbf1\{p_k\le\varepsilon\} \sim \varepsilon \ell_0(1/\varepsilon) \mbox{  as }\varepsilon\to0.
\end{eqnarray}
In this case, we necessarily have
$$
\ell(x) \sim \int_x^\infty u^{-1}\ell_0(u)\mathrm du \mbox{ as } x\to\infty
$$
and $\ell_0(x)/\ell(x)\to0$ as $x\to\infty$, see Proposition 15 in \cite{GHP07}. We will generally assume that \eqref{eq: for alpha 0} holds in this case.

\begin{pro}\label{prop:main RV limit}
Let $P=\{p_{k}\}_{k\ge 1}\in RV_\alpha(\ell)$. If $\alpha\in(0,1)$, then for all $r\ge 0$,
\begin{eqnarray}\label{eq:main RV limit}
\lim_{n\to\infty}\frac{{n\choose r}\sum_{k=1}^\infty p_{k}^{r+1} (1-p_{k})^{n-r}}{n^{\alpha-1}\ell(n)} = \frac{\alpha \Gamma(r+1-\alpha)}{r!}.
\end{eqnarray}
 If $\alpha=0$ and \eqref{eq: for alpha 0} holds, then for every $r\ge0$ 
\begin{eqnarray*}
\lim_{n\to\infty}\frac{\sum_{k=1}^\infty p^{r+1}_{k} (1-p_{k})^{n-r}}{n^{-1}\ell_0(n)} = 1.
\end{eqnarray*}
If $\alpha=1$, then for every $r\ge1$ the result in \eqref{eq:main RV limit} holds. 
If $\alpha=1$ and $r=0$ then
\begin{eqnarray*}
\lim_{n\to\infty}\frac{\sum_{k=1}^\infty p_{k} (1-p_{k})^{n}}{\ell_1(n)} = 1,
\end{eqnarray*}
where $\ell_1(x)=\int_x^\infty u^{-1}\ell(u)\mathrm du$ for $x>1$ is a function with $\ell_1\in\fS$ and $\ell(x)/\ell_1(x)\to0$ as $x\to\infty$.
\end{pro}

\begin{proof}
For $\alpha\in(0,1)$ this is Proposition 7 in \cite{OD12}.  For $\alpha=1$ the result follows by combining Proposition 18 in \cite{GHP07} with Lemma 2 in \cite{GZ17}. Similarly, for $\alpha=0$ the result follows by combining Proposition 19 in \cite{GHP07} with Lemma 2 in \cite{GZ17}. The facts about $\ell_1$ are given in Proposition 14 of \cite{GHP07}.
\end{proof}

\begin{pro}\label{prop:main RV limit to zero}
Fix $\alpha\in[0,1]$ and $\ell\in\fS$. When $\alpha\ne0$ assume that
\begin{eqnarray}\label{eq:RV to zero}
\lim_{\e\to 0}\frac{\nu(\e)}{\e^{-\alpha}\ell(1/\e)}=0
\end{eqnarray}
and when $\alpha=0$ assume that
\begin{eqnarray*}
\lim_{\e\to 0}\frac{\int_{[0,\e]} x \boldsymbol{\nu}_P(\rd x) }{ \varepsilon \ell(1/\e)} =0.
\end{eqnarray*}
If $\alpha\in[0,1)$, then for all $r\ge 0$,
\begin{eqnarray}\label{eq:RV to zero res}
\lim_{n\to\infty}\frac{{n\choose r}\sum_{k=1}^\infty p_{k}^{r+1} (1-p_{k})^{n-r}}{n^{\alpha-1}\ell(n)} = 0.
\end{eqnarray}
If $\alpha=1$ then for all $r\ge1$ the result in \eqref{eq:RV to zero res} holds. 
If $\alpha=1$ and $r=0$ then
\begin{eqnarray*}
\lim_{n\to\infty}\frac{\sum_{k=1}^\infty p_{k} (1-p_{k})^{n}}{\ell_1(n)} = 0,
\end{eqnarray*}
where $\ell_1(x)$ is derived from $\ell$ as in Proposition \ref{prop:main RV limit}.
\end{pro}

\begin{proof}
Let $q=r+1$, let $\boldsymbol\nu_P^q(\mathrm d x) = x^q \boldsymbol\nu_P(\mathrm d x)$, let $\Phi_q(n) = \frac{n^q}{q!} \sum_{k=1}^\infty p_k^q e^{-np_k}$, and note that
$$
\Phi_q(n) = \frac{n^q}{q!}\int_0^1 e^{-nx} \boldsymbol\nu_P^q(\mathrm d x).
$$
A standard application of Fubini's Theorem gives
$$
\boldsymbol\nu_P^q([0,s]) = \int_{[0,s]} x^q \boldsymbol\nu_P(\mathrm dx) =  q \int_0^s u^{q-1} \nu(u)\mathrm du - s^q\boldsymbol\nu_P((s,1]).
$$
Fix $\delta>0$. For $\alpha\ne0$, the assumptions imply that for small enough $s$ we have $\nu(s)\le \delta s^{-\alpha}\ell(1/s)$. It follows that for $\alpha\in(0,1)$ or $\alpha=1$ and $q\ge2$ we have
\begin{eqnarray*}
\lim_{s\to0^+} \frac{\boldsymbol\nu_P^q([0,s])}{s^{q-\alpha}\ell(1/s)} &=&  \lim_{s\to0^+} \frac{q \int_0^s u^{q-1} \nu(u)\mathrm du}{s^{q-\alpha}\ell(1/s)}
\le \delta  \lim_{s\to0^+} \frac{q \int_0^s u^{q-1-\alpha} \ell(1/u)\mathrm du}{s^{q-\alpha}\ell(1/s)} \\
&=&  \delta  \lim_{s\to0^+} \frac{q \int_{1/s}^\infty u^{-(q+1-\alpha)} \ell(u)\mathrm du}{s^{q-\alpha}\ell(1/s)} = \frac{q}{q-\alpha}\delta,
\end{eqnarray*}
where the last equality follows by Karamata's Theorem (Proposition 1.5.10 in \cite{Bingham:Goldie:Teugels:1987}).  Hence,
$$
\lim_{s\to0^+} \frac{\boldsymbol\nu_P^q([0,s])}{s^{q-\alpha}\ell(1/s)} = 0.
$$
Similarly, when $\alpha=1$ and $q=1$ we have
\begin{eqnarray*}
\lim_{s\to0^+} \frac{\nu^1_P([0,s])}{\ell_1(1/s)} &=&  \lim_{s\to0^+} \frac{ \int_0^s  \nu(u)\mathrm du}{\ell_1(1/s)}
\le \delta  \lim_{s\to0^+} \frac{\int_0^s u^{-1} \ell(1/u)\mathrm du}{\ell_1(1/s)} \\
&=&  \delta  \lim_{s\to0^+} \frac{ \int_{1/s}^\infty u^{-1} \ell(u)\mathrm du}{\ell_1(1/s)} = \delta  \lim_{s\to0^+} \frac{  \ell_1(1/s)}{\ell_1(1/s)} =\delta,
\end{eqnarray*}
and hence
$$
\lim_{s\to0^+} \frac{\nu^1_P([0,s])}{\ell_1(1/s)} =0.
$$
When $\alpha=0$ we have 
$$
\boldsymbol\nu_P^q([0,s]) = \int_{[0,s]} x^{q-1} \boldsymbol\nu_P^1(\mathrm dx) \le  s^{q-1}\boldsymbol\nu^1([0,s]),
$$
and hence
$$
\lim_{s\to0^+} \frac{\boldsymbol\nu_P^q([0,s]) }{s^q\ell(1/s)}\le \lim_{s\to0^+} \frac{s^{q-1}\boldsymbol\nu_P^1([0,s]) }{s^q\ell(1/s)} = \lim_{s\to0^+} \frac{\boldsymbol\nu_P^1([0,s]) }{s\ell(1/s)}=  0.
$$

From here a version of Karamata's Tauberian Theorem (Theorem 1.7.1' in \cite{Bingham:Goldie:Teugels:1987}) implies that for $\alpha\in[0,1)$ or $\alpha=1$ and $q\ge2$
\begin{eqnarray*}
\lim_{n\to\infty}\frac{\Phi_q(n)}{n^{\alpha}\ell(n)} = \frac{\int_0^1 e^{-nx} \boldsymbol\nu_P^q(\mathrm d x)}{q! n^{\alpha-q}\ell(n)}=0.
\end{eqnarray*}
and the corresponding result hold for the case $\alpha=1$ and $r=1$. 
From here, since $(n+1)^{\alpha-1}\ell(n+1)\sim n^{\alpha-1}\ell(n)$, and $\ell_1(n+1)\sim\ell_1(n)$, 
we can use Lemma 2 in \cite{GZ17} to complete the result.
\end{proof}

\section*{Acknowledgements}

The work of M.~Kelbert and Q.~Paris has been funded by the Russian Academic Excellence Project 5-100. The work of M.~Grabchak is supported, in part, by the Russian Science Foundation (Project \ftextnumero\ 17-11-01098).

\end{document}